\documentclass[12pt]{amsart}

\usepackage{amsmath, amssymb, graphics, setspace}
\usepackage[usenames]{color}
\input xypic

\usepackage{algorithm}
\usepackage[noend]{algpseudocode}

\algblock{Input}{EndInput}
\algnotext{EndInput}
\algblock{Output}{EndOutput}
\algnotext{EndOutput}

\renewcommand*\Call[2]{\textproc{#1}(#2)}

\makeatletter
\def\BState{\State\hskip-\ALG@thistlm}
\makeatother

\usepackage{verbatim}
\usepackage{url}
\usepackage{bm}
\usepackage{tikz}
\usepackage{algorithm}
\usepackage[noend]{algpseudocode}
\renewcommand*\Call[2]{\textproc{#1}(#2)}
\newcommand{\mathsym}[1]{{}}
\newcommand{\unicode}[1]{{}}

\makeatletter
\def\BState{\State\hskip-\ALG@thistlm}
\makeatother

\usepackage{amsmath}
\usepackage[centertags]{amsmath}
\usepackage{amsfonts}
\usepackage{amssymb}
\usepackage{amsthm}
\usepackage{newlfont}
\usepackage{url}
\usepackage{bm}
\usepackage{chngcntr}
\usepackage{tikz}
\usepackage{algorithm}
\usepackage{algpseudocode}
\theoremstyle{plain}
\newtheorem{thm}{Theorem}
\newtheorem{cor}[thm]{Corollary}

\newtheorem{lem*}[thm]{Lemma}
\newtheorem{prop}[thm]{Proposition}

\newtheorem{problem}[thm]{Problem}

\theoremstyle{definition}

\theoremstyle{remark}
\newtheorem{rem}{Remark}
\newtheorem{rem*}{Remark}

\newtheorem{example}[rem]{Example}

\numberwithin{rem}{section} 
\numberwithin{dfn}{section} 
\numberwithin{equation}{section} 
\numberwithin{thm}{section} 

\def\!{\operatorname{!}}

\def\1{\bold 1}

\def\Mat{\operatorname{Mat}}
\def\rank{\operatorname{rank}}
\def\Span{\operatorname{Span}}
\def\Col{\operatorname{Col}}

\newcommand{\Flag}{\bm{\mathrm{Flag}}}
\newcommand{\Required}{\bm{\mathrm{Required}}}
\newcommand{\Selected}{\bm{\mathrm{Selected}}}
\newcommand{\Basis}{\bm{\mathrm{Basis}}}
\newcommand{\OL}{\bm{\mathrm{OL}}}

\usepackage{amssymb}

\usepackage{enumitem}

\usepackage{graphicx}
\usepackage{xy}
\usepackage{float}

\makeatletter
\@namedef{subjclassname@2020}{%
	\textup{2020} Mathematics Subject Classification}
\makeatother

\usepackage[T1]{fontenc}
\usepackage{lmodern}

\newtheorem{theorem}{Theorem}[section]

\theoremstyle{definition}
\newtheorem{definition}[theorem]{Definition}
\newtheorem{remark}[theorem]{Remark}

\numberwithin{equation}{section}

\title{Compatible Bases and Actuator Scheduling}

\author{Dawid E. K{\k e}dzierski}
\address{ Institute of Mathematics, University of Szczecin,\\ ul. Wielkopolska 15, 70-451 Szczecin, Poland\\
}
\email{dawidedmundkedzierski@gmail.com}
\author{Piotr Kraso{\'n}}
\address{Department of Mathematics, Ariel University, Ariel 40700, Israel}
	\email{ piotrkras26@gmail.com}
\keywords{compatible bases, nested subspaces, rank profiles,
representative systems, actuator scheduling, controllability}

\subjclass[2020]{15A03, 93B03, 93B05}

\date{}

\usepackage[hidelinks,breaklinks=true]{hyperref}
\hypersetup{
  pdftitle={Compatible Bases and Actuator Scheduling},
  pdfauthor={Dawid E. Kedzierski and Piotr Krason}
}
\newcommand{\doi}[1]{\url{https://doi.org/#1}}

\begin{document}

\begin{abstract}
We study a constrained basis-selection problem for an ordered block
matrix \(B=\big[\, B_1\mid\cdots\mid B_n\,\big]\).  The successive tail matrices determine
a nested family of column spaces, and we consider bases of
representatives which are compatible with this induced flag.  Our main
linear-algebraic result shows that, among bases of representatives,
compatibility with the entire flag is characterized by the intrinsic
block quotas
\(r_k=w_k-w_{k+1}\), where \(w_k\) is the rank of the \(k\)-th tail
matrix.

We then show that the same structure arises naturally in
finite-horizon linear control.  For the reachability matrix
\(\big[\,A^{N-1}G\mid \cdots\mid G\,\big]\), requiring every actuator to be used exactly
once turns an actuator schedule into a representative system.
Preservation of the complete reachability profile is then equivalent
to compatibility with the tail-space flag.

This correspondence leads to an exact depth-first search with a
future-rank pruning criterion.  For an explicit structured family and
a fixed search ordering, the baseline search visits
\(3\cdot2^m+m-2\) nodes before the first solution, whereas future-rank
pruning reduces this number to \(3m+1\).  Exact-arithmetic experiments
reproduce these formulas for all tested instances.
\end{abstract}

\maketitle

\section{Introduction} 

Consider an ordered block matrix
\[
B=\big[\, B_1\mid \cdots\mid B_n \,\big].
\]
For each column position, the corresponding columns of the blocks form a family of possible representatives. The classical problem is to choose one vector from each family so that the selected vectors are linearly independent, as in the theory of independent transversals and Rado's theorem \cite{Rado1942}. Related representative-selection and matroid structures are discussed in \cite{EdmondsFulkerson1965,Hall1935, Oxley2011,Schrijver2003}.

The ordered block structure leads naturally to the tail matrices
\[
B_{\ge k}=\big[\, B_k\mid \cdots \mid B_n \,\big]
\]
and their nested column spaces
\[
W_n\subseteq W_{n-1}\subseteq\cdots\subseteq W_1.
\]
We study bases of representatives compatible with this flag, meaning that for every \(k\), the representatives selected from \(B_k,\ldots,B_n\) span \(W_k\).

Let
\[
w_k=\dim W_k,\qquad r_k=w_k-w_{k+1},
\]
with \(w_{n+1}=0\). Our main linear-algebraic result shows that, among bases of representatives, compatibility with the whole flag is equivalent to the simple block-count condition that exactly \(r_k\) vectors are selected from \(B_k\).

This structure has a direct interpretation in finite-horizon linear control. For
\[
x_{t+1}=Ax_t+Gu_t,
\]
the reachability matrix
\[
\big[\, A^{N-1}G\mid A^{N-2}G\mid \cdots \mid G\, \big]
\]
has the same block form. If every actuator is used exactly once, choosing its activation time amounts to choosing one representative from the corresponding column family. We show that schedules preserving the complete tail reachability profile are precisely the compatible bases above. In particular, this requirement is stronger than terminal controllability. Related questions of controllability, sparse actuator scheduling, and actuator placement have been studied in \cite{BallottaJosephThete2024,Joseph2024,KondapiEtAl2024,KondapiEtAl2025, Sontag1998,
SummersCortesiLygeros2016}.

The same characterization yields an exact depth-first search. In addition to block quotas and linear independence, a future-rank test rejects partial selections that can no longer attain the required tail ranks. For an explicit structured family and a fixed search ordering, the baseline search visits
\[
3\cdot 2^m+m-2
\]
nodes before the first solution, whereas future-rank pruning reduces this number to
\[
3m+1.
\]
Exact-arithmetic experiments reproduce these formulas for all tested instances.

\section{Rectangular tail matrices and rank profiles}
\label{sec:tails}

Let \(L\) be a field and let \(V=L^d\).  Consider a block matrix
\(B=\big[\, B_1\mid B_2\mid\cdots\mid B_n\, \big]\), where
\(B_j\in\Mat_d(L)\).  For each \(j=1,\ldots,n\), write
\(B_j=\big[b_j^1,b_j^2,\ldots,b_j^d\big]\), where \(b_j^i\in V\)
denotes the \(i\)-th column of \(B_j\).

\begin{definition}
For \(k=1,\ldots,n\), the \emph{tail matrix} of level \(k\) is
\[
 B_{\geq k}:=\big[\, B_k\mid B_{k+1} \mid \cdots \mid B_n \,\big].
\]
\end{definition}

Thus \(B_{\geq k}\in\Mat_{d\times d(n-k+1)}(L)\).  In particular,
the natural matrices governing the problem are rectangular, even
though the individual blocks \(B_j\) are square.

Set
\(w_k:=\rank B_{\geq k}\) for \(k=1,\ldots,n\), and
\(w_{n+1}:=0\).  Then
\(0=w_{n+1}\leq w_n\leq\cdots\leq w_1\leq d\).

\begin{definition}
The sequence \((w_1,\ldots,w_n)\) is called the
\emph{tail rank profile} of \(B\).
\end{definition}

For \(k=1,\ldots,n\), put \(r_k:=w_k-w_{k+1}\).  Then
\(r_k\geq0\) and
\(\sum\limits_{k=1}^n r_k=w_1=\rank B\).  In particular, if \(B\) has
full row rank, then \(\sum\limits_{k=1}^n r_k=d\).

For \(k=1,\ldots,n\), let
\(W_k:=\Col(B_{\geq k})\) and set \(W_{n+1}:=\{0\}\).
Equivalently,
\(W_k=\Span\big\{b_j^i:j\geq k,\ 1\leq i\leq d\big\}\).
Since \(B_{\geq k+1}\) is obtained from \(B_{\geq k}\) by deleting
the block \(B_k\), these spaces form the flag
\begin{equation}
\label{eq:flag}
 \{0\}=W_{n+1}
 \subseteq W_n
 \subseteq W_{n-1}
 \subseteq\cdots
 \subseteq W_1
 \subseteq V.
\end{equation}
Moreover, \(\dim W_k=w_k\), and hence
\[
 r_k
 =
 \dim W_k-\dim W_{k+1}
 =
 \dim\big(W_k/W_{k+1}\big).
\]
Thus \(r_k\) measures the increase in dimension contributed at level
\(k\) relative to the subsequent tail.

\section{Compatible bases of representatives}
\label{sec:compatible}

For each column position \(i=1,\ldots,d\), consider the indexed family
\[
 \mathcal V^i=\big\{(j,b_j^i):j=1,\ldots,n\big\}.
\]
Thus the block index from which a representative is selected is
regarded as part of the selection data. A representative system chooses
one indexed vector \((j_i,b_{j_i}^i)\) from each family
\(\mathcal V^i\).

\begin{definition}
A basis
\(\mathcal B=\big(b_{j_1}^1,\ldots,b_{j_d}^d\big)\) of \(V\) is called a
\emph{basis of representatives compatible with the flag}
\eqref{eq:flag} if, for every \(k=1,\ldots,n\),
\[
 \Span\big(\mathcal B_{\geq k}\big)=W_k,
 \qquad
 \mathcal B_{\geq k}
 :=
 \big\{b_{j_s}^s\in\mathcal B:j_s\geq k \big\}.
\]
\end{definition}

We always regard the indices \(j_1,\ldots,j_d\) as part of the data
of the representative system. The tail rank profile determines
automatically how many representatives must come from each block.

\begin{prop}
\label{prop:required}
Let \(\mathcal B=\big(b_{j_1}^1,\ldots,b_{j_d}^d \big)\) be a basis of
representatives compatible with the flag. Then, for every
\(k=1,\ldots,n\),
\[
 \#\big\{i:j_i=k \big\}=w_k-w_{k+1}=r_k.
\]
\end{prop}

\begin{proof}
Since \(\mathcal B\) is a basis, every subfamily is linearly
independent. Compatibility therefore gives
\(\#\mathcal B_{\geq k}=\dim W_k=w_k\) and
\(\#\mathcal B_{\geq k+1}=w_{k+1}\).
Their difference is precisely the number of vectors selected from
the \(k\)-th block, which proves the assertion.
\end{proof}

\begin{prop}[Quota characterization]
\label{prop:quota-characterization}
Let \(\mathcal B=\big(b_{j_1}^1,\ldots,b_{j_d}^d\big)\) be a basis of
representatives. Then \(\mathcal B\) is compatible with the
tail-space flag if and only if
\[
 \#\big\{i:j_i=k\big\}=r_k,
 \qquad k=1,\ldots,n.
\]
\end{prop}

\begin{proof}
Necessity follows from Proposition~\ref{prop:required}.

Conversely, suppose that
\(\#\{i:j_i=k\}=r_k\) for every \(k\). Fix \(k\). 
Then
\[
 \#\mathcal B_{\geq k}
 =
 \sum_{j=k}^n r_j
 =
 w_k.
\]
Since \(\mathcal B_{\geq k}\) is a subfamily of the basis
\(\mathcal B\), it is linearly independent, and hence
\(\dim\Span\big(\mathcal B_{\geq k}\big)=w_k\).

Every vector in \(\mathcal B_{\geq k}\) belongs to one of the blocks
\(B_k,\ldots,B_n\). Hence
\[
 \Span\big(\mathcal B_{\geq k}\big)\subseteq W_k.
\]
Since both spaces have dimension \(w_k\), they are equal.
 Thus
\(\Span\big(\mathcal B_{\geq k} \big)=W_k\) for every \(k\), and
\(\mathcal B\) is compatible with the tail-space flag.
\end{proof}
\begin{remark}
The numbers \(r_k\) are determined entirely by the matrix \(B\).
They specify how many representatives a compatible basis must select
from each block \(B_k\). Conversely, once a representative system is
known to be a basis, these block counts are sufficient for
compatibility: it is compatible with the tail-space flag if and only if
it selects exactly \(r_k\) vectors from \(B_k\) for every \(k\).
\end{remark}

Assume now that \(\rank B=d\), and set
\[
 k_0:=\max\{k:w_k=d\}.
\]

\begin{prop}[Trimming]
\label{prop:trimming}
Let \(k_0=\max\big\{k:w_k=d\big\}\). Then
\(r_1=\cdots=r_{k_0-1}=0\). Consequently, no compatible basis of
representatives contains a vector from
\(B_1,\ldots,B_{k_0-1}\), and the compatible-basis problem for \(B\)
is equivalent to the same problem for the trimmed matrix
\[
 B_{\mathrm{trim}}:=B_{\geq k_0}.
\]
\end{prop}

\begin{proof}
By definition of \(k_0\),
\(w_1=\cdots=w_{k_0}=d\), so
\(r_j=w_j-w_{j+1}=0\) for every \(j<k_0\).
By Proposition~\ref{prop:quota-characterization}, a compatible basis
therefore uses no representatives from
\(B_1,\ldots,B_{k_0-1}\). These blocks may consequently be deleted
without changing either the existence or the set of compatible bases.
\end{proof}

Thus \(B_{\mathrm{trim}}\) is the shortest tail of \(B\) having full
row rank.

\begin{problem}[Compatible representative basis]
\label{problem:compatiblebasis}
Given a block matrix \(B\) of full row rank, determine whether it
admits a basis of representatives compatible with its tail-space flag,
and construct one whenever it exists.
\end{problem}

\section{Single-use actuator scheduling}
\label{sec:control}

We now show how the compatible-basis problem arises naturally in a
finite-horizon control setting.

Consider the discrete-time linear system
\begin{equation}
\label{eq:controlsystem}
 x_{t+1}=Ax_t+Gu_t,
 \qquad t=0,\ldots,N-1,
\end{equation}
where \(A\in\mathbb R^{d\times d}\) and
\(G=\big[g_1,\ldots,g_d\big]\in\mathbb R^{d\times d}\).
We write \(u_t=\big(u_t^1,\ldots,u_t^d\big)^T\).

The matrix \(A\) describes the evolution of the state during one time
step.  The matrix \(G\) is the actuator matrix: its \(i\)-th column
\(g_i\) is the state direction directly affected by actuator \(i\),
while \(u_t^i\) is the scalar input applied to this actuator at time
\(t\).

For reachability we take \(x_0=0\). Iterating
\eqref{eq:controlsystem} gives
\[
 x_N=\sum_{j=1}^{N}A^{N-j}G\,u_{j-1}.
\]
We call \(j\) the \(j\)-th \emph{control slot}. Thus slot \(j=1\)
corresponds to time \(t=0\), whereas slot \(j=N\) corresponds to the
last control action, at time \(t=N-1\).

If actuator \(i\) is used in slot \(j\) with scalar amplitude
\(\alpha\), its contribution to the terminal state is
\[
 \alpha A^{N-j}g_i.
\]
Hence an actuator used earlier is propagated through more powers of
\(A\): in the first slot its terminal direction is \(A^{N-1}g_i\),
whereas in the last slot it is simply \(g_i\).

For each slot \(j\), define
\begin{equation}
\label{eq:controlblocks}
 B_j:=A^{N-j}G.
\end{equation}
Thus the \(i\)-th column of \(B_j\) is \(A^{N-j}g_i\), the terminal
direction produced by using actuator \(i\) in slot \(j\). The full
finite-horizon reachability matrix is therefore
\[
 B=\big[\, A^{N-1}G\mid A^{N-2}G\mid\cdots\mid G\, \big].
\]
In this sense, \(B\) is a catalogue of all actuator--slot terminal
directions available over the horizon.

For \(k=1,\ldots,N\), the tail space
\[
 W_k
 =
 \Col\big[\,A^{N-k}G\mid A^{N-k-1}G\mid\cdots\mid G\,\big]
\]
has a direct control interpretation: it is the space of terminal-state
increments that can still be generated when only the control slots
\(k,k+1,\ldots,N\) remain available. Thus \(W_1\) is the finite-horizon
reachable space, whereas \(W_N=\operatorname{Col}(G)\).
We write \(w_k=\dim W_k\) and assume throughout this section that
\(w_1=d\). Equivalently, the system is reachable over the prescribed
horizon \(N\).

\subsection{Single-use schedules}

We now impose the restriction that every actuator is used exactly
once during the horizon.

\begin{definition}
A \emph{single-use actuator schedule} is a map
\[
 \tau:\{1,\ldots,d\}\longrightarrow\{1,\ldots,N\},
\]
where \(\tau(i)\) is the control slot in which actuator \(i\) is used.
Different actuators may be assigned to the same slot.
\end{definition}

The schedule specifies when an actuator is used, but not its scalar
amplitude. If actuator \(i\) is assigned amplitude \(\alpha_i\), then
the input in slot \(j\) is
\[
 u_{j-1}
 =
 \sum_{\tau(i)=j}\alpha_i e_i,
\]
where \(e_i\) is the \(i\)-th standard basis vector. Consequently,
the terminal state generated by the schedule is
\[
 x_N
 =
 \sum_{i=1}^d
 \alpha_i A^{N-\tau(i)}g_i.
\]

Thus the schedule selects the \(d\) terminal directions
\begin{equation}
\label{eq:schedule-directions}
 \mathcal B(\tau)
 =
 \big(
 A^{N-\tau(1)}g_1,\ldots,
 A^{N-\tau(d)}g_d
 \big)
 =
 \big(
 b_{\tau(1)}^1,\ldots,b_{\tau(d)}^d
 \big).
\end{equation}

\begin{definition}
A single-use schedule \(\tau\) is called
\emph{terminally controllable} if the selected directions span
\(\mathbb R^d\), equivalently,
\[
 \rank
 \big[\, A^{N-\tau(1)}g_1\ \cdots\ A^{N-\tau(d)}g_d\,\big]
 =d.
\]
\end{definition}

Terminal controllability concerns only the state that can be reached
at the end of the horizon. We shall require more: the schedule should
not exhaust the available control directions too early.

At the beginning of slot \(k\), the actuators with
\(\tau(i)<k\) have already been used, while those with
\(\tau(i)\geq k\) are still available. Their remaining terminal
directions span
\[
 \Span\big\{
 A^{N-\tau(i)}g_i:\tau(i)\geq k
 \big\}.
\]
If no single-use schedule is imposed, the terminal directions
available from slots \(k,\ldots,N\) span \(W_k\).

\begin{definition}
A single-use actuator schedule \(\tau\) is called
\emph{reachability-profile preserving} if
\begin{equation}
\label{eq:profilepreserving}
 \Span\big\{
 A^{N-\tau(i)}g_i:\tau(i)\geq k
 \big\}
 =
 W_k
\end{equation}
for every \(k=1,\ldots,N\).
\end{definition}

Thus a profile-preserving schedule retains at every stage exactly the
same future terminal reachability space as the unrestricted system.
For \(k=1\), this condition gives \(W_1=\mathbb R^d\), so every
profile-preserving schedule is terminally controllable.

\begin{example}
\label{ex:two-actuators}
Let \(d=2\), \(N=3\), and consider
\[
 A=
 \begin{bmatrix}
 1&0\\
 1&1
 \end{bmatrix},
 \qquad
 G=I_2=[g_1,g_2].
\]
Thus \(g_1=(1,0)^T\) and \(g_2=(0,1)^T\).

The two columns of \(G\) describe the instantaneous directions of the
two actuators. The matrix \(A\) describes how these directions evolve
between consecutive control slots. In particular,
\(Ag_1=(1,1)^T\) and \(Ag_2=(0,1)^T\).

Since \(N=3\), the three slot matrices are
\[
 B_1=A^2G=
 \begin{bmatrix}
 1&0\\
 2&1
 \end{bmatrix},
 \qquad
 B_2=AG=
 \begin{bmatrix}
 1&0\\
 1&1
 \end{bmatrix},
 \qquad
 B_3=G=
 \begin{bmatrix}
 1&0\\
 0&1
 \end{bmatrix}.
\]
The columns of \(B_1\) are the terminal effects of firing the two
actuators in the first slot, those of \(B_2\) correspond to the second
slot, and those of \(B_3\) to the final slot. Thus
\[
 B=\big[\,B_1\mid B_2\mid B_3\,\big]= \big[\,A^2G\mid AG\mid G\,\big]
\]
collects all possible actuator--slot choices.

Consider the schedule
\[
 \tau(1)=1,\qquad \tau(2)=3.
\]
It selects the first column of \(B_1\) and the second column of \(B_3\),
namely \(A^2g_1=(1,2)^T\) and \(g_2=(0,1)^T\).
With amplitudes \(\alpha_1,\alpha_2\),
\[
 x_3
 =
 \alpha_1(1,2)^T+\alpha_2(0,1)^T.
\]
These two directions are linearly independent, so the schedule is
terminally controllable.

It is not, however, reachability-profile preserving. At the beginning
of slot \(2\), actuator \(1\) has already been used, and only actuator
\(2\) remains available. Hence the remaining selected directions span
only \(\Span\{g_2\}\). On the other hand, the unrestricted system still
has both actuators available in slots \(2\) and \(3\), and
\(W_2=\mathbb R^2\). Therefore
\[
 \Span\{g_2\}\neq W_2.
\]

By contrast, the schedule
\[
 \tau(1)=\tau(2)=3
\]
keeps both actuators available until the last slot. Their terminal
directions are then \(g_1\) and \(g_2\), and the schedule preserves
the complete reachability profile.
\end{example}

\begin{remark}
A concrete physical picture is an idealized small satellite equipped
with two one-shot thrusters. The columns of \(G\) describe the
instantaneous state changes produced by the two thrusters, while
\(A\) describes how these changes propagate through the subsequent
dynamics.

Firing a thruster in an early slot may strongly influence the terminal
state because its effect is propagated for several time steps. Once
fired, however, that thruster is no longer available for later
corrections. In Example~\ref{ex:two-actuators}, firing actuator \(1\)
in the first slot still allows full terminal controllability, but it
removes one control direction from the remaining part of the horizon.
The profile-preserving condition prevents this premature loss of
future control authority.
\end{remark}

\subsection{Equivalence with compatible bases}

The preceding construction identifies the scheduling problem directly
with the representative-basis problem.

\begin{thm}
\label{thm:equivalence}
Let \(B_j=A^{N-j}G\). There is a one-to-one correspondence between
reachability-profile preserving single-use actuator schedules and
bases of representatives compatible with the tail-space flag of \(B\).
The correspondence is
\[
 \tau
 \longmapsto
 \mathcal B(\tau)
 =
 \big(b_{\tau(1)}^1,\ldots,b_{\tau(d)}^d \big).
\]
\end{thm}

\begin{proof}
A schedule chooses, for each actuator \(i\), exactly one indexed vector
\(b_{\tau(i)}^i\), and hence determines a representative system.

If the schedule is profile preserving, the condition for \(k=1\)
shows that the selected \(d\) vectors span \(\mathbb R^d\), and hence
form a basis. For every \(k\), the vectors corresponding to actuators
scheduled in slots \(k,\ldots,N\) are precisely
$
 \{b_{\tau(i)}^i:\tau(i)\geq k\}.
$
Therefore condition \eqref{eq:profilepreserving} is exactly the
compatibility condition with the tail space \(W_k\). The converse is
identical.
\end{proof}
The quota characterization of Section~\ref{sec:compatible} now gives
a simple description of all profile-preserving schedules.
\begin{cor}[Schedule characterization]
\label{cor:schedule-characterization}
A single-use actuator schedule \(\tau\) is reachability-profile
preserving if and only if
\begin{enumerate}[label=(\roman*)]
\item it is terminally controllable, that is,
\[
 \rank
 \big[\, A^{N-\tau(1)}g_1\ \cdots\ A^{N-\tau(d)}g_d\, \big]
 =d;
\]
\item exactly \(r_k=w_k-w_{k+1}\) actuators are assigned to slot
\(k\), that is,
\[
 \#\big\{i:\tau(i)=k \big\}=r_k,
 \qquad k=1,\ldots,N.
\]
\end{enumerate}
\end{cor}

\begin{proof}
By Theorem~\ref{thm:equivalence}, the schedule corresponds to the
representative system \(\mathcal B(\tau)\). Condition (i) says that
this system is a basis, while condition (ii) gives the block quotas.
The result follows from
Proposition~\ref{prop:quota-characterization}.
\end{proof}

Thus \(r_k\) is exactly the number of actuators that a
reachability-profile preserving schedule must assign to control
slot \(k\).

\section{Preprocessing and candidate construction}
\label{sec:preprocessing}

We first reduce the search to the shortest tail having full row rank.
Starting from \(w_{n+1}=0\), we compute the tail ranks
\[
 w_j=\rank B_{\geq j},
 \qquad j=n,n-1,\ldots,1,
\]
and the corresponding block quotas
\[
 r_j=w_j-w_{j+1}.
\]
If \(\rank B<d\), no compatible basis exists. Otherwise, let
\[
 k_0=\max\big\{k:w_k=d\big\}.
\]
By Proposition~\ref{prop:trimming}, only the blocks
\(B_{k_0},\ldots,B_n\) need to be considered.

For each representative family \(i=1,\ldots,d\), define the candidate
list
\[
 L_i=
 \big\{(i,j,b_j^i):
   k_0\leq j\leq n,\;
   b_j^i\neq0\big\}.
\]
Within each list, candidates are considered in decreasing block index.
If some \(L_i\) is empty, no basis of representatives exists.

The candidate lists are ordered by increasing cardinality,
\[
 \OL=(L_{i_1},\ldots,L_{i_d}),
 \qquad
 \#L_{i_1}\leq\cdots\leq\#L_{i_d},
\]
with ties resolved by preserving the original order. This is a
fail-first ordering: representative families with fewer admissible
candidates are processed earlier.

\subsection{Baseline depth-first search}

We retain the basic depth-first search as a reference algorithm.
At each level it checks whether the block quota has been exhausted and
whether the new vector is independent of the previously selected
vectors.

\begin{algorithm}[H]
\caption{Baseline depth-first selection}
\label{alg:baseline}
\begin{algorithmic}[1]

\Procedure{DFSSelection}
{$\OL,\Required,k_0,n$}

\State \(d\gets\#\OL\)

\State
\(\Basis\gets
(\mathrm{None},\ldots,\mathrm{None})\)

\State
\(\Selected\gets
(\mathrm{None},\ldots,\mathrm{None})\)

\For{\(j=k_0,\ldots,n\)}
    \State \(\mathrm{Used}[j]\gets0\)
\EndFor

\Procedure{DFS}{$\ell$}

\If{\(\ell>d\)}
    \State \Return SUCCESS
\EndIf

\ForAll{\((i,m,c)\in\OL[\ell]\)}

    \If{\(\mathrm{Used}[m]\geq\Required[m]\)}
        \State \textbf{continue}
    \EndIf

    \If{\(c\notin\Span(\Basis[1:\ell-1])\)}

        \State \(\Basis[\ell]\gets c\)

        \State
        \(\Selected[\ell]\gets(i,m,c)\)

        \State
        \(\mathrm{Used}[m]\gets
        \mathrm{Used}[m]+1\)

        \State
        \(result\gets\Call{DFS}{\ell+1}\)

        \If{\(result=\mathrm{SUCCESS}\)}
            \State \Return \(\mathrm{SUCCESS}\)
        \EndIf

        \Statex \Comment{Backtrack}

        \State
        \(\mathrm{Used}[m]\gets
        \mathrm{Used}[m]-1\)

        \State
        \(\Basis[\ell]\gets\mathrm{None}\)

        \State
        \(\Selected[\ell]\gets\mathrm{None}\)

    \EndIf

\EndFor

\State \Return \(\mathrm{FAILURE}\)

\EndProcedure

\State \(result\gets\Call{DFS}{1}\)

\If{\(result=\mathrm{SUCCESS}\)}
    \State \Return \((\mathrm{True},\Selected)\)
\Else
    \State \Return \((\mathrm{False},\varnothing)\)
\EndIf
\EndProcedure
\end{algorithmic}
\end{algorithm}
Here \(\Required[k]=r_k\).

The complete baseline procedure consists of the preprocessing and
candidate construction described above, followed by
\textsc{DFSSelection}. If the search succeeds, the selected triples
\((i,m,c)\) determine the compatible basis by placing \(c\) in
representative position \(i\).

\section{Reachability-aware pruning}
\label{sec:modified}

The baseline search uses only local information: block quotas and
linear independence. We now use the tail spaces to detect partial
selections that cannot be completed.

Let \(S\) be a partial selection. For each \(k\), let \(S_{\geq k}\)
denote the vectors already selected from blocks \(B_k,\ldots,B_n\).
Let \(I(S)\) be the set of unassigned representative families, and let
\(u_j(S)\) be the number of vectors already selected from \(B_j\).

For \(k=1,\ldots,n\), define
\[
 R_{\geq k}(S)
 =
 \big\{b_j^i:
   i\in I(S),\;
   j\geq k,\;
   u_j(S)<r_j,\;
   b_j^i\neq0\big\}.
\]
Thus \(R_{\geq k}(S)\) collects all candidates that are individually
admissible for extending \(S\) within the tail
\(B_k,\ldots,B_n\). Since it may contain several candidates from the
same unassigned representative family, its span can be larger than the
space attainable by an actual completion.

\begin{example}
Let \(d=4\), \(n=3\), with
\[
 (r_1,r_2,r_3)=(1,1,2),
 \qquad
 (w_1,w_2,w_3)=(4,3,2),
\]
and
\[
 B_1=[0,0,e_4,0],\qquad
 B_2=[e_3,0,e_1+e_2,e_2],\qquad
 B_3=[e_1,e_2,0,0].
\]
For the partial selection
\[
 S=\{b_3^1,b_3^2\}=\{e_1,e_2\},
\]
we have \(I(S)=\{3,4\}\) and \(u_3(S)=2=r_3\). Hence, at level
\(k=2\),
\[
 R_{\geq2}(S)
 =
 \{b_2^3,b_2^4\}
 =
 \{e_1+e_2,e_2\}.
\]
Therefore
\[
 \dim\Span\bigl(S_{\geq2}\cup R_{\geq2}(S)\bigr)
 =2<3=w_2,
\]
so the partial selection cannot be completed.
\end{example}

\subsection{Future-rank feasibility}

\begin{prop}[Future-rank pruning]
If, for some \(k\),
\begin{equation}
\label{eq:future-rank}
 \dim\Span\bigl(S_{\geq k}\cup R_{\geq k}(S)\bigr)
 <w_k,
\end{equation}
then \(S\) cannot be extended to a compatible basis.
\end{prop}

\begin{proof}
For any completion \(\mathcal B\) of \(S\),
\[
 \mathcal B_{\geq k}
 \subseteq
 S_{\geq k}\cup R_{\geq k}(S).
\]
Hence
\[
 \dim\Span(\mathcal B_{\geq k})
 \leq
 \dim\Span\bigl(S_{\geq k}\cup R_{\geq k}(S)\bigr).
\]
If the latter dimension is smaller than
\(w_k=\dim W_k\), then
\(\Span(\mathcal B_{\geq k})\neq W_k\).
\end{proof}

\begin{remark}
The future-rank condition is necessary but not sufficient for
extendability, since \(R_{\geq k}(S)\) may contain several candidates
from the same unassigned representative family. Thus
\[
 \dim\Span\bigl(S_{\geq k}\cup R_{\geq k}(S)\bigr)
\]
is only an upper bound on the tail rank attainable by a completion.
\end{remark}

In the control interpretation, a partial schedule may therefore be
discarded whenever the already selected and still admissible
actuator--slot directions cannot span one of the required tail
reachability spaces \(W_k\).

\subsection{Residual capacity}

The remaining block quotas give another necessary condition. Define
\[
 q_j(S):=r_j-u_j(S),
\]
and, for \(J\subseteq\{k_0,\ldots,n\}\), let
\(\Gamma_S(J)\) be the set of unassigned representative families having
an admissible candidate in a block belonging to \(J\). Any extendable
partial selection satisfies
\begin{equation}
\label{eq:hallcondition}
 \sum_{j\in J}q_j(S)
 \leq
 |\Gamma_S(J)|
 \qquad
 \text{for every }J\subseteq\{k_0,\ldots,n\}.
\end{equation}
This is a Hall-type residual capacity condition. Equivalently, residual
capacity may be tested by a bipartite matching or maximum-flow problem.
It provides an additional safe pruning criterion.

\subsection{Future-feasibility procedure}

The implemented feasibility test combines the block-capacity check
with the future-rank condition.

\begin{algorithm}[H]
\caption{Future-feasibility test}
\label{alg:futurefeasible}
\begin{algorithmic}[1]

\Procedure{FutureFeasible}
{$\Selected,\mathrm{Used},\OL,
\Flag,\Required,k_0,n$}

\For{\(j=k_0,\ldots,n\)}
    \If{\(\mathrm{Used}[j]>\Required[j]\)}
        \State \Return \(\mathrm{False}\)
    \EndIf
\EndFor

\For{\(k=n,n-1,\ldots,k_0\)}

    \State construct \(S_{\geq k}\) from the selected triples

    \State construct \(R_{\geq k}(S)\) from the unassigned candidate lists

    \If{
    \(\rank\bigl(S_{\geq k}\cup R_{\geq k}(S)\bigr)
      <\Flag[k]\)
    }
        \State \Return \(\mathrm{False}\)
    \EndIf

\EndFor

\State \Return \(\mathrm{True}\)

\EndProcedure

\end{algorithmic}
\end{algorithm}
Here \(\Flag[k]=w_k\) and \(\Required[k]=r_k\).

\subsection{Modified depth-first search}

The reachability-aware search is obtained from the baseline search by
applying Algorithm~\ref{alg:futurefeasible} after every tentative
selection. Recursion proceeds only if the enlarged partial selection
passes this test; otherwise the algorithm backtracks immediately.
All candidate ordering, quota, independence, and backtracking steps are
unchanged.

Thus the baseline and modified searches differ only by the
future-feasibility test. For the control problem, a selected triple
\((i,j,c)\) assigns actuator \(i\) to slot \(j\), so the resulting
schedule is given by
\[
 \tau(i)=j.
\]

\section{Correctness}
\label{sec:correctness}

We now verify that neither the baseline search nor the
reachability-aware pruning changes the set of solutions.

\begin{thm}
\label{thm:correctness}
The baseline algorithm and the modified algorithm return a compatible
basis if and only if such a basis exists.
\end{thm}

\begin{proof}
Suppose first that either algorithm returns a complete selection.
Exactly one vector has then been chosen from each representative
family, and the independence test guarantees that the resulting
\(d\) vectors form a basis.

For every block \(B_j\), the search enforces the upper bound \(r_j\).
Since
\[
 \sum_{j=k_0}^{n} r_j=d,
\]
a complete selection of \(d\) vectors satisfying these bounds must
use exactly \(r_j\) vectors from each block \(B_j\). By
Proposition~\ref{prop:quota-characterization}, the resulting basis is
therefore compatible with the tail-space flag.

Conversely, suppose that a compatible basis \(\mathcal B\) exists.
By Proposition~\ref{prop:trimming}, it uses only the blocks
\(B_{k_0},\ldots,B_n\), so all of its vectors occur in the candidate
lists considered by the search. Hence \(\mathcal B\) determines a
branch of the search tree.

No partial selection along this branch is rejected by the independence
test, since every subfamily of a basis is linearly independent.
Moreover, Proposition~\ref{prop:required} shows that the branch never
exceeds any block quota.

It remains to consider the additional future-rank test used by the
modified algorithm. Let \(S\) be a partial selection lying on the
branch determined by \(\mathcal B\). For any \(k\), every vector of
\(\mathcal B_{\geq k}\) is either already contained in
\(S_{\geq k}\) or belongs to an unassigned representative family in a
block whose quota has not yet been filled. Hence it is contained in
\(R_{\geq k}(S)\). Therefore
\[
 \mathcal B_{\geq k}
 \subseteq
 S_{\geq k}\cup R_{\geq k}(S).
\]
Since \(\mathcal B\) is compatible,
\(\Span(\mathcal B_{\geq k})=W_k\), and consequently
\[
 \dim\Span\bigl(S_{\geq k}\cup R_{\geq k}(S)\bigr)
 \geq w_k.
\]
Thus the future-rank test never rejects this branch.

The branch corresponding to \(\mathcal B\) therefore remains in the
finite search tree and is eventually reached by both algorithms.
\end{proof}

Combining Theorem~\ref{thm:correctness} with
Theorem~\ref{thm:equivalence} gives the corresponding statement for
the control problem.

\begin{cor}
For
\(B=\big[\, A^{N-1}G\mid A^{N-2}G\mid\cdots\mid G\,\big]\), the modified algorithm
returns a solution if and only if there exists a
reachability-profile preserving single-use actuator schedule.
\end{cor}

\section{Complexity and search reduction}

Let \(N_i=\#L_i\). After ordering the candidate lists as
\[
 \OL=(L_{i_1},\ldots,L_{i_d}),
 \qquad
 N_{i_1}\leq\cdots\leq N_{i_d},
\]
the search tree has depth \(d\). Without pruning, the number of nodes
at depth \(\ell\) is at most
\[
 \prod_{q=1}^{\ell}N_{i_q},
\]
and hence the total number of recursive nodes is bounded by
\[
 1+\sum_{\ell=1}^{d}\prod_{q=1}^{\ell}N_{i_q}.
\]
Thus the exact search remains combinatorial in the worst case.

The baseline search prunes only when a block quota is exhausted or a
linear dependence occurs. The future-rank test may reject a branch
earlier whenever
\[
 \dim\Span\bigl(S_{\geq k}\cup R_{\geq k}(S)\bigr)<w_k
\]
for some \(k\). This does not improve the general worst-case bound, but
can substantially reduce the number of visited nodes for structured
instances. Section~10 gives an explicit family exhibiting such a
reduction.

\section{A three-dimensional example}

Consider
\[
 A_0=
 \begin{bmatrix}
 0&1&0\\
 0&0&1\\
 -1&1&0
 \end{bmatrix},
 \qquad
 G_0=
 \begin{bmatrix}
 1&0&1\\
 0&1&1\\
 0&0&0
 \end{bmatrix},
 \qquad N=3.
\]
Writing \(G_0=[g_1,g_2,g_3]\), the reachability blocks are
\[
\begin{aligned}
 B_1&=A_0^2G_0=
 \begin{bmatrix}
 0&0&0\\
 -1&1&0\\
 0&-1&-1
 \end{bmatrix},
&
 B_2&=A_0G_0=
 \begin{bmatrix}
 0&1&1\\
 0&0&0\\
 -1&1&0
 \end{bmatrix},
&
 B_3&=G_0=
 \begin{bmatrix}
 1&0&1\\
 0&1&1\\
 0&0&0
 \end{bmatrix}.
\end{aligned}
\]
Their tail-rank and quota profiles are
\[
 (w_1,w_2,w_3)=(3,3,2),
 \qquad
 (r_1,r_2,r_3)=(0,1,2).
\]

The schedule
\[
 \tau_{\mathrm{good}}=(3,2,3)
\]
selects \(g_1,A_0g_2,g_3\). These vectors are linearly independent,
and their selected tail ranks are \((3,3,2)\); hence the schedule is
reachability-profile preserving.

By contrast,
\[
 \tau_{\mathrm{bad}}=(1,2,3)
\]
selects \(A_0^2g_1,A_0g_2,g_3\). These vectors are again linearly
independent, so the schedule is terminally controllable, but its
selected tail-rank profile is
\[
 (3,2,1)\neq(3,3,2).
\]
Equivalently, its block-count vector is
\((1,1,1)\), rather than the required \((0,1,2)\).
Thus terminal controllability does not imply preservation of the
complete reachability profile.

\section{An explicit exponential-separation family}
\label{sec:exponential-family}

We now use the system of Section~9
to construct a family for which future-rank pruning reduces the number
of nodes visited before the first solution from exponential to linear
growth.

The counts below refer to the deterministic ordering used by the
algorithm: candidate lists are ordered by increasing cardinality,
ties preserve the original actuator order, and candidates within each
list are examined in decreasing block index.

\begin{prop}
\label{prop:exponential-separation}
For \(m\geq1\), let
\[
 A_m=I_m\otimes A_0,
 \qquad
 \widetilde G_m=I_m\otimes G_0,
\]
and obtain \(G_m\) from \(\widetilde G_m\) by the column permutation
\[
 \pi_m=
 (1,4,\ldots,3m-2,\;
  2,5,\ldots,3m-1,\;
  3,6,\ldots,3m).
\]
Thus the actuators are ordered by type: first all type--\(1\)
actuators, then all type--\(2\), and finally all type--\(3\).

For horizon \(N=3\), set \(B_j=A_m^{3-j}G_m\). Then
\[
 (w_1,w_2,w_3)=(3m,3m,2m),
 \qquad
 (r_1,r_2,r_3)=(0,m,2m).
\]
With the ordering above, the baseline search visits
\[
 N_{\mathrm{base}}(m)=3\cdot2^m+m-2
\]
nodes before finding its first compatible basis, whereas the search
with future-rank pruning visits
\[
 N_{\mathrm{future}}(m)=3m+1.
\]
\end{prop}

\begin{proof}
Write
\[
 V=V_1\oplus\cdots\oplus V_m,
 \qquad
 V_r=\Span\{e_{r,1},e_{r,2},e_{r,3}\}.
\]
In the \(r\)-th summand, the three actuator vectors and their images
under \(A_m\) are
\[
\begin{aligned}
 g_{r,1}&=e_{r,1},
 & A_mg_{r,1}&=-e_{r,3},\\
 g_{r,2}&=e_{r,2},
 & A_mg_{r,2}&=e_{r,1}+e_{r,3},\\
 g_{r,3}&=e_{r,1}+e_{r,2},
 & A_mg_{r,3}&=e_{r,1}.
\end{aligned}
\]

The vectors in block \(3\) span
\(\Span\{e_{r,1},e_{r,2}\}\) in each summand, so \(w_3=2m\).
Adding block \(2\) supplies the missing direction \(e_{r,3}\), hence
\(W_2=V\) and \(w_2=3m\). Also \(w_1=3m\). Therefore
\[
 (w_1,w_2,w_3)=(3m,3m,2m),
 \qquad
 (r_1,r_2,r_3)=(0,m,2m),
\]
and \(k_0=2\).

We first consider the baseline search. Because the actuators are
ordered by type, the first \(m\) representative families are of
type~\(1\). Their first candidates are the block--\(3\) vectors
\(e_{r,1}\), which are mutually independent and belong to the branch
leading to the first solution. Hence exactly one node is visited at
each depth \(0,\ldots,m\).

For a type--\(2\) actuator, the two candidates in the trimmed problem
are
\[
 e_{r,2}\quad\text{from block }3,
 \qquad
 e_{r,1}+e_{r,3}\quad\text{from block }2.
\]
Both are independent of the previously selected \(e_{r,1}\), and at
this stage neither choice violates a block quota. Thus each
type--\(2\) actuator doubles the number of surviving branches. After
\(s\) such actuators, there are \(2^s\) nodes at depth \(m+s\), so
\[
 N_k=2^{k-m},
 \qquad m<k\leq2m.
\]

For a type--\(3\) actuator, the candidates are
\[
 e_{r,1}+e_{r,2}\quad\text{from block }3,
 \qquad
 e_{r,1}\quad\text{from block }2.
\]
The block--\(2\) candidate is always dependent on the previously
selected type--\(1\) vector \(e_{r,1}\). The block--\(3\) candidate
is admissible exactly when the corresponding type--\(2\) actuator was
chosen from block \(2\). Indeed,
\[
 \{e_{r,1},e_{r,2},e_{r,1}+e_{r,2}\}
\]
is dependent, whereas
\[
 \{e_{r,1},e_{r,1}+e_{r,3},e_{r,1}+e_{r,2}\}
\]
is independent.

Each processed type--\(3\) actuator therefore eliminates one half of
the remaining branches. Hence
\[
 N_k=2^{3m-k},
 \qquad 2m<k\leq3m.
\]
The complete depth profile is
\[
 N_k=
 \begin{cases}
 1, & 0\leq k\leq m,\\
 2^{k-m}, & m<k\leq2m,\\
 2^{3m-k}, & 2m<k\leq3m.
 \end{cases}
\]
Summing over the depths gives
\[
\begin{aligned}
 N_{\mathrm{base}}(m)
 &=(m+1)
   +\sum_{s=1}^{m}2^s
   +\sum_{s=1}^{m}2^{m-s} \\
 &=3\cdot2^m+m-2.
\end{aligned}
\]

Now consider future-rank pruning. Suppose that a type--\(2\) actuator
in \(V_r\) is tentatively chosen from block \(3\). The selected
directions in this summand are then \(e_{r,1}\) and \(e_{r,2}\).
The only remaining actuator is of type~\(3\), whose possible vectors
are \(e_{r,1}\) and \(e_{r,1}+e_{r,2}\). Thus every direction still
available in \(V_r\) lies in
\(\Span\{e_{r,1},e_{r,2}\}\), so the attainable contribution of this
summand to the level--\(2\) rank is at most \(2\).

The remaining \(m-1\) summands contribute at most \(3(m-1)\)
dimensions. Therefore the total attainable rank is at most
\[
 2+3(m-1)=3m-1<w_2=3m.
\]
The future-rank test rejects the block--\(3\) choice immediately.

Consequently, every type--\(2\) actuator is forced into block \(2\).
After these choices, the type--\(3\) block--\(3\) candidates are
independent of the previously selected vectors. The search therefore
follows a single recursive branch through depths \(0,\ldots,3m\), and
hence
\[
 N_{\mathrm{future}}(m)=3m+1.
\]
\end{proof}

\begin{remark}
A visited node means a call to the recursive DFS procedure. A
candidate rejected by a quota, independence, or future-rank test does
not create a new visited node.
\end{remark}

\begin{remark}
The formulas in Proposition~\ref{prop:exponential-separation} refer to
the fixed actuator and candidate ordering above and to the search for
the first solution. They are not order-independent worst-case bounds.
They exhibit a structured family for which future-rank pruning changes
the first-solution search tree from exponential to linear size.
\end{remark}

Since \(d=3m\), the same counts may be written as
\[
 N_{\mathrm{base}}(d)
 =3\cdot2^{d/3}+\frac d3-2,
 \qquad
 N_{\mathrm{future}}(d)=d+1.
\]
Thus, for this family and the fixed search ordering, future-rank
pruning changes the number of visited nodes before the first solution
from exponential growth in \(d\) to linear growth.

\section{Computational experiments}
\label{sec:experiments}

We compare the baseline depth-first search with the version using future-rank pruning on the family of Section 10. All computations were carried out in exact arithmetic using the TailBasis package [7]. The package also contains the code for analyzing the examples and reproducing the computational experiments presented in this paper. A visited node means one call to the recursive depth-first search procedure.

For \(m=1,\ldots,8\), the observed numbers of visited nodes are shown
in Table~\ref{tab:nodecounts}. Recall that the state-space dimension is
\(d=3m\).

\begin{table}[ht]
\centering
\begin{tabular}{c|r|r|r}
\(m\) & Baseline & Future-rank & Future-rank prunings\\
\hline
1 & 5   & 4  & 1\\
2 & 12  & 7  & 2\\
3 & 25  & 10 & 3\\
4 & 50  & 13 & 4\\
5 & 99  & 16 & 5\\
6 & 196 & 19 & 6\\
7 & 389 & 22 & 7\\
8 & 774 & 25 & 8
\end{tabular}
\vspace{14pt}
\caption{Visited nodes and future-rank pruning events for the
benchmark family.}
\label{tab:nodecounts}
\end{table}

The observed counts agree exactly with Proposition~10.1:
\[
 N_{\mathrm{base}}(m)=3\cdot2^m+m-2,
 \qquad
 N_{\mathrm{future}}(m)=3m+1.
\]
For example, at \(m=8\), hence \(d=24\), the baseline search visits
\(774\) nodes, whereas the future-rank search visits \(25\).

The observed depth profiles also agree with the formulas derived in
Section~10. For \(m=4\), the baseline search visits
\[
 1,1,1,1,1,2,4,8,16,8,4,2,1
\]
nodes at successive depths, whereas the future-rank search visits one
node at every depth. Thus the computations reproduce both the total
node counts and the structure of the search tree predicted by
Proposition~10.1.

\section{Further directions}

The present formulation assumes \(d\) actuator channels in a
\(d\)-dimensional state space and requires every actuator to be used
exactly once. Several natural extensions remove these restrictions.

One may allow more actuator channels than state dimensions and seek
both a subset of actuators and their activation times. A second
possibility is to replace the single-use condition by prescribed upper
bounds on the number of activations of each actuator. Both variants
introduce additional combinatorial constraints into the representative
selection problem.

Another direction is optimization. Once feasible
reachability-profile preserving schedules are known to exist, one may
minimize a prescribed cost depending on activation times, actuator
usage, or control amplitudes. The present depth-first search can then
serve as the basis of a branch-and-bound procedure.

\section{Conclusion}

We studied a constrained representative-selection problem for a block
matrix \(B=[B_1|\cdots|B_n]\) through its rectangular tail matrices
\(B_{\geq k}\). Their column spaces form a nested family, while the
rank differences \(r_k=w_k-w_{k+1}\) determine the exact number of
representatives that a compatible basis must select from each block.
This reduces compatibility with the entire tail-space flag to a basis
condition together with intrinsic block quotas.

For finite-horizon linear control, the same structure arises from the
reachability matrix
\([A^{N-1}G|\cdots|G]\). A single-use actuator schedule preserves the
complete reachability profile precisely when its selected terminal
directions form a compatible basis of representatives. This
interpretation leads naturally to an exact depth-first search in which
the tail ranks are used not only to determine the required block
counts, but also to prune partial selections that can no longer attain
the required future ranks.

For the explicit family constructed in
Section~\ref{sec:exponential-family}, the number of visited nodes
before the first solution changes from
\[
 N_{\mathrm{base}}(d)
 =3\cdot2^{d/3}+\frac{d}{3}-2,
 \qquad
 N_{\mathrm{future}}(d)=d+1.
\]
The computational experiments reproduce these counts for all tested
dimensions \(d=3,6,\ldots,24\), illustrating that tail-rank information
can substantially reduce the search tree when the underlying
reachability structure is sufficiently rigid.



%
%

\end{document}